\documentclass[reqno,11pt]{amsart}

\usepackage[all]{xy}
\usepackage{amssymb}
\usepackage{amsgen}
\usepackage{amsmath}
\usepackage{amsthm}
\usepackage{mathrsfs}
\usepackage{cite}
\usepackage{amsfonts}

\newcommand{\Bis}{\mathop{\mathrm{Bis}}\nolimits}
\newcommand{\skel}[1]{^{(#1)}}
\newcommand{\Is}{\mathop{\mathrm{Is}}}
\newcommand{\dom}{\mathop{\boldsymbol d}}
\newcommand{\ran}{\mathop{\boldsymbol r}}

\renewcommand{\to}{\longrightarrow}

\newcommand{\Ind}{\mathop{\mathrm{Ind}}\nolimits}

\newcommand{\inv}{^{-1}}
\newcommand{\p}{\varphi}

\newcommand{\ov}[1]{\ensuremath{\overline {#1}}}

\newcommand{\wh}{\widehat}

\usepackage{xcolor}

\newtheorem{Thm}{Theorem}[section]
\newtheorem{Prop}[Thm]{Proposition}

\newtheorem{Lemma}[Thm]{Lemma}
{\theoremstyle{definition}
}
{\theoremstyle{remark}
}
\newtheorem{Cor}[Thm]{Corollary}
{\theoremstyle{remark}
}
{\theoremstyle{remark}
\newtheorem{Example}[Thm]{Example}}

{\theoremstyle{remark}
}
{\theoremstyle{remark}
}
{\theoremstyle{remark}
}

\numberwithin{equation}{section}

\title[\'Etale groupoid algebras]{Simplicity, primitivity and semiprimitivity of \'etale groupoid algebras with applications to inverse semigroup algebras}

\author{Benjamin Steinberg}
\address[B.~Steinberg]{%
    Department of Mathematics\\
    City College of New York\\
    Convent Avenue at 138th Street\\
    New York, New York 10031\\
    USA}
\email{bsteinberg@ccny.cuny.edu}

\thanks{This work was partially supported by a grant from the Simons Foundation(\#245268
to Benjamin Steinberg) and the Binational Science Foundation of Israel and the US (\#2012080 to Benjamin Steinberg)}
\date{\today}


\begin{document}

\begin{abstract}
This paper studies simplicity, primitivity and semiprimitivity of algebras associated to \'etale groupoids. Applications to inverse semigroup algebras are presented.  The results also recover the semiprimitivity of Leavitt path algebras and can be used to recover the known primitivity criterion for Leavitt path algebras.
\end{abstract}

\maketitle

\section{Introduction}
The author~\cite{mygroupoidalgebra,mygroupoidarxiv} associated to a commutative ring with unit $\Bbbk$ and an \'etale groupoid $\mathscr G$ with locally compact  Hausdorff and totally disconnected unit space a $\Bbbk$-algebra $\Bbbk\mathscr G$, which is a discrete analogue of the groupoid $C^*$-algebra of $\mathscr G$~\cite{Renault,Paterson,Exel}.  This class of algebras includes group algebras, inverse semigroup algebras and Leavitt path algebras.  Further study of these algebras has occurred in~\cite{operatorguys1,operatorsimple1,operatorsimple2,groupoidbundles,GroupoidMorita}. Recently, Nekrashevych~\cite{Nekgrpd} has used \'etale groupoid algebras to construct finitely generated simple algebras of quadratic growth over any base field.

In this paper we study simplicity, primitivity and semiprimitivity of group\-oid algebras.  In particular, an extension of the famous result of Amitsur on semiprimitivity of group algebras in characteristic $0$ over fields which are not algebraic over $\mathbb Q$ is obtained for groupoid algebras. Applications are then presented to inverse semigroup algebras.  In particular, we recover results of Munn~\cite{MunnSemiprim,MunnSemiprim2,MunnAlgebraSurvey} and Domanov~\cite{Domanov} on semiprimitivity of inverse semigroup algebras with simpler, and more conceptual proofs.  A primitivity result of Munn~\cite{Munnprimitive} is also recovered.  We give a partial answer to a question of Munn as to which contracted inverse semigroup algebras are simple~\cite{MunnAlgebraSurvey}. Namely, we characterize all inverse semigroups with Hausdorff universal groupoid whose contracted semigroup algebra is simple.  This includes all $0$-$E$-unitary inverse semigroups.

We also note that the semiprimitivity of Leavitt path algebras over any base field~\cite{Leavittsemiprimitivity} is an immediate consequence of our results.  The primitivity results for Leavitt path algebras of~\cite{Leavittprimitivity} can also be rederived from our results, although we do not do so here.  The main observation is that condition (L) corresponds to effectiveness of the corresponding \'etale groupoid and that the further condition needed for primitivity amounts to the groupoid having a dense orbit.

The paper is organized as follows.   First we recall basics on \'etale groupoids, inverse semigroups and their associated algebras.  Then we discuss the extension of the simplicity/uniquness results of~\cite{operatorsimple1} to arbitrary rings (see the historical discussion below for connections  with~\cite{operatorsimple2}).  This is followed by our main results on primitivity and semiprimitivity.  We then discuss inverse semigroup algebras and how topological properties of groupoids of germs relate to dynamical properties of inverse semigroup actions (see the historical discussion below for connections with~\cite{ExelPardo}).  Then our \'etale groupoid results are applied to inverse semigroup algebras.

\subsubsection*{Historical note}
This paper began when the author read~\cite{operatorsimple1} and realized it could be use to make progress on an old question of Munn~\cite{MunnAlgebraSurvey}.  The author was able to remove the assumption that the base field was $\mathbb C$ and obtained results on minimality and effectiveness of tight groupoids of inverse semigroups.  These results were announced at the workshop ``Semigroups and Applications" (Uppsala, August 2012). Since then the simplicity results were obtained independently by Clark and Edie-Michell~\cite{operatorsimple2} (and submitted before we wrote up our results). The results on primitivity and semiprimitivity were obtained in 2013/2014 and presented at ``Partial Actions and Representations Symposium" (Gramado, May 2014) and the ``Fields Institute Workshop on Groups, Rings and Group Rings" (July 2014).   As we were finalizing this paper for submission, Exel and Pardo placed on ArXiv the paper~\cite{ExelPardo}, which contains quite a bit of overlap with our results on tight groupoids that were announced in Uppsala and Gramado, and which appear in the final section of this paper. The work in~\cite{ExelPardo} was obtained independently of our work and some of it was mentioned in the meeting at Gramado in connection with self-similar group actions.

\section{Groupoids, inverse semigroups and their algebras}
This section contains preliminaries about groupoids, inverse semigroups and their algebras.  Lawson~\cite{Lawson} is the definitive reference for inverse semigroups theory. For \'etale groupoids, we recommend~\cite{Renault,Exel,Paterson}.  Algebras of ample groupoids were introduced in~\cite{mygroupoidalgebra}; see also~\cite{mygroupoidarxiv} for some additional results not included in~\cite{mygroupoidalgebra} as well as~\cite{operatorguys1,operatorsimple1,operatorsimple2}.

\subsection{Inverse semigroups}
An \emph{inverse semigroup} is a semigroup $S$ such that, for all $s\in S$, there exists unique $s^*\in S$ with $ss^*s=s$ and $s^*ss^*=s^*$.  Notice that $s^*s,ss^*$ are idempotents. Also, note that $(st)^*=t^*s^*$.  Idempotents of $S$ commute and so $E(S)$ is a subsemigroup.  Moreover, it is a meet semilattice with respect to the ordering $e\leq f$ if $ef=e$.  In fact, $S$ itself is ordered by $s\leq t$ if $s=te$ for some idempotent $e\in E(S)$, or equivalently $s=ft$ for some $f\in E(S)$.  This partial order is compatible with multiplication and stable under the involution. We put $s^{\downarrow}=\{t\in S\mid t\leq s\}$ and $s^{\uparrow}=\{t\in S\mid t\geq s\}$. If $e\in E(S)$, then $G_e=\{s\in S\mid s^*s=e=ss^*\}$ is a group called the \emph{maximal subgroup} of $S$ at $e$.  It is the group of units of the monoid $eSe$.

All groups are inverse semigroups, as are all (meet) semilattices. A semidirect product $E\rtimes G$ of a group $G$ and a semilattice $E$ is also an inverse semigroup. If $X$ is a topological space, then the set of all homeomorphisms between open subsets of $X$ is an inverse semigroup $I_X$ under the usual composition of partial functions.  An inverse semigroup $S$ has a zero element $z$, if $zs=z=sz$ for all $s\in S$.  Zero elements are unique when they exist and will often be denoted $0$.  The zero element of $I_X$ is the empty partial bijection.

By an action of an inverse semigroup $S$ on a space $X$, we mean a homomorphism $\theta\colon S\to I_X$ such that if we put $X_e=\mathrm{dom}(\theta(e))$, then  \[\bigcup_{e\in E(S)}X_e= X.\]  This last condition
is a non-degeneracy condition and implies, for instance, that a group must act by homeomorphisms.  We write $\mathrm{Fix}(s)$ for the fixed-point set of $s$ and we put
\begin{equation}\label{defineXs}
X_s=\bigcup_{\{e\in E(S)\mid e\leq s\}}X_e.
\end{equation}
Note that if $s$ is idempotent, then both definitions of $X_e$ agree and so there is no ambiguity.  Trivially, $X_s\subseteq \mathrm{Int}(\mathrm{Fix}(s))$ because $s$ fixes $X_e$ pointwise if $e\leq s$ is an idempotent.

\begin{Prop}\label{p:faithful}
If $S$ is an inverse semigroup acting faithfully on a space $X$ and if the $X_e$ with $e\in E(S)$ form a basis for the topology on $X$, then $X_s=\mathrm{Int}(\mathrm{Fix}(s))$.
\end{Prop}
\begin{proof}
Clearly, $X_s\subseteq\mathrm{Fix}(s)$ and since $X_s$ is open, it consists of interior points. Conversely, let $x$ be an interior point $\mathrm{Fix}(s)$ and suppose that $X_e$ is a basic neighborhood of $x$ with $X_e\subseteq \mathrm{Fix}(s)\subseteq X_{s^*s}$.  Then we deduce that $e\leq s^*s$ by faithfulness of the action and that $sex=x=ex$ for all $x\in X_e=X_{(se)^*(se)}$.  Thus $se=e$, that is, $e\leq s$, by faithfulness. Therefore, $x\in X_s$.
\end{proof}

A \emph{congruence} on an inverse semigroup $S$ is an equivalence relation $\equiv$ such that $s\equiv s'$ implies $us\equiv us'$ and $sv\equiv s'v$ for all $u,v\in S$.   An inverse semigroup $S$ is called \emph{congruence-free} if the only congruences on $S$ are the equality relation and the universal relation.  For example, a group is congruence-free if and only if it is simple.  We consider neither the trivial inverse semigroup, nor the empty inverse semigroup to be congruence-free.

An inverse semigroup $S$ is \emph{$E$-unitary} if $s\geq e$ with $e\in E(S)$ implies that $s\in E(S)$.  An inverse semigroup $S$ with zero is called \emph{$0$-$E$-unitary} (or \emph{$E^*$-unitary}) if $s\geq e$ with $e\in E(S)\setminus \{0\}$ implies $s\in E(S)$.  We shall say that an inverse semigroup $S$ is \emph{Hausdorff} if, for all $s,t\in S$, the set $s^\downarrow\cap t^\downarrow$ is finitely generated as a lower set, that is, there is a finite set $F$ such that $x\leq s,t$ if and only if $x\leq u$ for some $u\in F$. (The term weak semilattice is used for this in~\cite{mygroupoidalgebra}.) It is known that $E$-unitary and $0$-$E$-unitary inverse semigroups are Hausdorff~\cite{mygroupoidalgebra} (in fact, this follows directly from Proposition~\ref{p:Hausdorffsemichar} below). The reason for the terminology Hausdorff will become apparent later.

\begin{Prop}\label{p:Hausdorffsemichar}
An inverse semigroup $S$ is Hausdorff if and only if the lower set $(s^*s)^\downarrow\cap s^\downarrow$ of $E(S)$ is  finitely generated for all $s\in S$.
\end{Prop}
\begin{proof}
Necessity is clear.  For sufficiency, suppose that $s,t\in S$ and put $u=ts^*$.  Then we claim that the mapping $e\mapsto es$ provides an order isomorphism $(u^*u)^{\downarrow}\cap u^\downarrow\to s^\downarrow\cap t^\downarrow$ with inverse $x\mapsto xs^*$. The proposition will then follow.  From $u^*u= st^*ts^*$ we have that $e\leq u^*u,u$ implies that $es\leq us=ts^*s\leq t$ and $es\leq s$ (because $e$ is idempotent).  Conversely, if $x\leq s,t$, then $xs^*\leq ss^*,ts^*$ and hence $xs^*\leq u$ and is an idempotent.  Thus $xs^*\leq u^*u$.  It remains to show these are inverse mappings.  Note that $e\leq u$ implies that $e=ts^*f$ for some idempotent $f$. Then $ess^*=ts^*fss^*=ts^*ss^*f=ts^*f=e$.  Conversely, if $x\leq s,t$, then $x=sf$ with $f\in E(S)$ and so $xs^*s=sfs^*s=sf=x$.  This completes the proof.
\end{proof}

If $\Bbbk$ is a commutative ring with unit, then the semigroup algebra $\Bbbk S$ of an inverse semigroup $S$ is defined in the usual way as a $\Bbbk$-algebra with basis $S$ and multiplication extending that of $S$ via the distributive law. If $S$ is an inverse semigroup with zero element $z$, then the contracted semigroup algebra is $\Bbbk_0S=\Bbbk S/\Bbbk z$.  It amounts to identifying the zero of $S$ with the zero of $\Bbbk$ and it is universal for representations of $S$ into $\Bbbk$-algebras that preserve the zero element.

Occasionally, we shall require the notion of a \emph{generalized boolean algebra}, that is a relatively complemented, distributive lattice with bottom.  Generalized boolean algebras are, up to isomorphism, ideals in boolean algebras.

\subsection{\'Etale groupoids}
In this paper, compactness will include the Hausdorff axiom.  However, we do not require locally compact spaces to be Hausdorff. A topological groupoid $\mathscr G=(\mathscr G\skel 0,\mathscr G\skel 1)$ is \emph{\'etale} if its domain map $\dom$ (or, equivalently, its range map $\ran$) is a local homeomorphism.  In this case, identifying objects with identity arrows, we have that $\mathscr G\skel 0$ is an open subspace of $\mathscr G\skel 1$ and the multiplication map is a local homeomorphism.  Details can be found in~\cite{Paterson,resendeetale,Exel}.

Following~\cite{Paterson}, an \'etale groupoid is called \emph{ample} if its unit space $\mathscr G\skel 0$ is locally compact Hausdorff with a basis of compact open subsets. We shall say that an ample groupoid $\mathscr G$ is Hausdorff if $\mathscr G\skel 1$ is Hausdorff.

A \emph{local bisection} of an \'etale groupoid $\mathscr G$ is an open subset $U\subseteq \mathscr G\skel 1$ such that $\dom|_U$ and $\ran|_U$ are homeomorphisms.  The local bisections form a basis for the topology on $\mathscr G\skel 1$~\cite{Exel}. The set $\Bis(\mathscr G)$ of local bisections is an inverse monoid under the binary operation \[UV = \{uv\mid u\in U,\ v\in V,\ \dom (u)=\ran (v)\}.\] The semigroup inverse is given by $U^* = \{u\inv\mid u\in U\}$ and $E(\Bis(\mathscr G))=\Bis(\mathscr G\skel 0)$. The inverse monoid $\Bis(\mathscr G)$ acts on $\mathscr G\skel 0$ by partial homeomorphisms by putting \[U\cdot x=\begin{cases}y, & \text{if there is $g\in U$ with $\dom(g)=x,\ran(g)=y$}\\ \text{undefined},  & \text{else.}\end{cases}\] The set $\Bis_c(\mathscr G)$ of compact open local bisections (which should be thought of as the set of local bisections with compact support) is an inverse subsemigroup of $\Bis(\mathscr G)$ (it is a submonoid if and only if $\mathscr G\skel 0$ is compact)~\cite{Paterson}. Note that $\mathscr G$ is ample if and only if $\Bis_c(\mathscr G)$ is a basis for the topology on $\mathscr G\skel 1$~\cite{Exel,Paterson}.

The \emph{isotropy subgroupoid} of a groupoid $\mathscr G=(\mathscr G\skel 0,\mathscr G\skel 1)$ is the subgroupoid $\Is(\mathscr G)$ with $\Is(\mathscr G)\skel 0=\mathscr G\skel 0$ and \[\Is(\mathscr G)\skel 1=\{g\in \mathscr G\skel 1\mid \dom(g)=\ran(g)\}.\]  The \emph{isotropy group} of $x\in \mathscr G\skel 0$ is the group \[G_x=\{g\in \mathscr G\skel 1\mid \dom(g)=x=\ran(g)\}.\] An \'etale groupoid is said to be \emph{effective} if $\mathscr G\skel 0=\mathrm{Int}(\Is(\mathscr G)\skel 1)$.  It is well known, and easy to prove, that an ample groupoid $\mathscr G$ is effective if and only if the natural action of $\Bis_c(\mathscr G)$ on $\mathscr G\skel 0$ is faithful.

If $x\in \mathscr G\skel 0$, then the \emph{orbit} $\mathcal O_x$ of $x$ consists of all $y\in \mathscr G\skel 0$ such that there is an arrow $g$ with $\dom(g)=x$ and $\ran(g)=y$.  The orbits form a partition of $\mathscr G\skel 0$. If $\mathscr G$ is ample, then the orbits of $\mathscr G$ are precisely the orbits for the natural action of $\Bis_c(\mathscr G)$ on $\mathscr G\skel 0$.

A subset $X\subseteq \mathscr G\skel 0$ is \emph{invariant} if it is a union of orbits.
Equivalently, $X$ is invariant if and only if it is invariant under the natural action of $\Bis_c(\mathscr G)$ on $\mathscr G\skel 0$.  An \'etale groupoid is said to be \emph{minimal} if $\mathscr G\skel 0$ has no proper, non-empty closed invariant subsets, or equivalently, if each orbit is dense.

A key example of an \'etale groupoid is that of a groupoid of germs.  Let $S$ be an inverse semigroup acting on a locally compact Hausdorff space $X$.  The groupoid of germs $\mathscr G=S\ltimes X$ is defined as follows. One puts $\mathscr G\skel 0=X$ and $\mathscr G\skel 1=\{(s,x)\in S\times X\mid x\in X_{s^*s}\}/{\sim}$ where $(s,x)\sim (t,y)$ if and only if $x=y$ and there exists $u\leq s,t$ with $x\in X_{u^*u}$. Note that if $S$ is a group, then there are no identifications. The $\sim$-class of an element $(s,x)$ is denoted $[s,x]$.  The topology on $\mathscr G\skel 1$ has basis all sets of the form $(s,U)$ where $U\subseteq X_{s^*s}$ is open and $(s,U) = \{[s,x]\mid x\in U\}$.  Note that if $[t,x]\in (s,U)$, then $(t,V)\subseteq (s,U)$ for some open neighborhood $V$ with $x\in V\subseteq U$. Indeed, since $[t,x]=[s,x]$, there exists $u\leq s,t$ with $x\in X_{u^*u}$.  It follows that if $V=U\cap X_{u^*u}\cap X_{t^*t}$, then $x\in V$ and $[t,y]=[s,y]$ for all $y\in V$.  Thus each arrow $[t,x]$ has a basis of neighborhoods of the form $(t,U)$ with $U\subseteq X_{t^*t}$.   One puts $\dom([s,x])=x$, $\ran([s,x])=sx$ and defines $[s,ty][t,y]=[st,y]$.  Inversion is given by $[s,x]\inv = [s^*,sx]$.  Note that $(s,X_{s^*s})\in \Bis(S\ltimes X)$ and if $X_{s^*s}$ is compact, then $(s,X_{s^*s})\in \Bis_c(S\ltimes X)$.  See~\cite{Exel,Paterson} for details.

The following criterion generalizes a result from~\cite{mygroupoidalgebra}.  It was first observed by the author (unpublished) under the assumption that the domains were compact open and then it was observed by R.~Exel and E.~Pardo, that clopen suffices (private communication).

\begin{Prop}\label{p:hausdorffcondition}
Let $S$ be an inverse semigroup acting on a locally compact Hausdorff space $X$ and suppose that $X_e$ is clopen for all $e\in E(S)$.  Then $S\ltimes X$ is Hausdorff if and only if, for each $s\in S$, the set $X_s$ is closed.  In particular, if $S$ is Hausdorff, then so is $S\ltimes X$.
\end{Prop}
\begin{proof}
Note that $X_s\subseteq X_{s^*s}$.
Suppose first that $S\ltimes X$ is Hausdorff and let $x\in X\setminus X_s$.  If $x\notin X_{s^*s}$, then $X\setminus X_{s^*s}$ is a neighborhood of $x$ disjoint from $X_s$.  So assume that $x\in X_{s^*s}$.  Then we claim that $[s,x]\neq [s^*s,x]$. Indeed, if $(s,x)\sim (s^*s,x)$, then we can find $u\leq s,s^*s$ with $x\in X_{u^*u}$.  But $u\leq s^*s$ implies that $u\in E(S)$ and $X_u=X_{u^*u}$.  Therefore, we have $x\in X_s$, a contradiction.  Thus we can find disjoint basic neighborhoods $(s,U)$ and $(s^*s,V)$ of $[s,x]$ and $[s^*s,x]$ respectively. Then $U\cap V$ is a neighborhood of $x$ and we claim it is disjoint from $X_s$. Indeed, if $y\in U\cap V\cap X_s$, then $y\in X_e$ for some idempotent $e\leq s$.  But then $e\leq s,s^*s$ and $y\in X_e$, whence $[s,y]=[s^*s,y]\in (s,U)\cap (s^*s,V)$, a  contradiction. We conclude $X_s$ is closed.

Conversely, suppose that $X_s$ is closed for all $x\in X$.  Let $[s,x]\neq [t,y]$ be arrows of $S\ltimes X$.  If $x\neq y$, then we can choose disjoint neighborhoods $U,V$ of $x,y$ in $X$, respectively. We may assume without loss of generality that $U\subseteq X_{s^*s}$ and $V\subseteq X_{t^*t}$.  Then $(s,U)$ and $(t,V)$ are disjoint neighborhoods of $[s,x]$ and $[t,y]$ respectively.

So assume next that $x=y$ and  put $u=s^*t$. We claim that $x\notin X_u$.  Indeed, if $x\in X_u$, then $x\in X_e$ for some idempotent $e\leq s^*t$. Put $z=se$. and write $e=s^*tf$ with $f\in E(S)$. Then $z^*z=s^*se=s^*ss^*tf=s^*tf=e$ and so $x\in X_{z^*z}$.  Clearly $z\leq s$.  But since $e\leq u$, we have $z=se\leq su=ss^*t\leq t$.  Thus $[s,x]=[t,x]$, a contradiction.  Since $X_u$ is closed, there is a neighborhood $U$ of $x$ with $U\subseteq X\setminus X_u$.  Without loss of generality, we may assume that $U\subseteq X_{s^*s}\cap X_{t^*t}$.

We claim that $(s,U)$ and $(t,U)$ are disjoint neighborhoods of $[s,x]$ and $[t,x]$, respectively. Indeed, suppose that $[s,x']=[t,x']$ belongs to $(s,U)\cap (t,U)$.  Then there exists $w\in S$ with $x'\in X_{w^*w}$ and $w\leq s,t$.  Then $w^*w\leq s^*t=u$ and so $x'\in X_u\cap U$, a contradiction.  This completes the proof of the first statement.

For the final statement, observe that the union in \eqref{defineXs} is finite if $S$ is Hausdorff (because $\{e\in E(S)\mid e\leq s\}=(s^*s)^{\downarrow} \cap s^{\downarrow}$) and so $X_s$ is closed.
\end{proof}

We remark that since $X_s\subseteq X_{s^*s}$, if $X_{s^*s}$ is compact, then $X_s$ is closed if and only if it is compact.

If $X$ is a Hausdorff space with a basis of compact open sets, then $S\ltimes X$ will be an ample groupoid~\cite{mygroupoidalgebra}.

\subsection{\'Etale groupoid algebras}

Fix now a commutative ring with unit $\Bbbk$.  The author~\cite{mygroupoidalgebra} associated a $\Bbbk$-algebra $\Bbbk \mathscr G$ to each ample groupoid $\mathscr G$ as follows.  We define $\Bbbk\mathscr G$ to be the $\Bbbk$-span in $\Bbbk^{\mathscr G\skel 1}$ of the characteristic functions $\chi_U$ of compact open subsets $U$ of $\mathscr G\skel 1$.  It is shown in~\cite[Proposition~4.3]{mygroupoidalgebra} that $\Bbbk \mathscr G$ is spanned by the elements $\chi_U$ with $U\in \Bis_c(\mathscr G)$.  If $\mathscr G\skel 1$ is Hausdorff, then $\Bbbk \mathscr G$ consists of the continuous $\Bbbk$-valued functions on $\mathscr G\skel 1$ with compact support where $\Bbbk$ is endowed with the discrete topology.  Convolution is defined on $\Bbbk \mathscr G$ by \[\p\ast \psi(g)=\sum_{\dom(h)=\dom(g)}\p(gh\inv)\psi(h).\]  The finiteness of this sum is proved in~\cite{mygroupoidalgebra}. The fact that the convolution belongs to $\Bbbk \mathscr G$ rests on the computation $\chi_U\ast \chi_V=\chi_{UV}$ for $U,V\in \Bis_c(\mathscr G)$~\cite{mygroupoidalgebra}.

Note that since $\mathscr G\skel 0$ is open in $\mathscr G\skel 1$, it follows that $\Bbbk \mathscr G\skel 0$ (where we view $\mathscr G\skel 0$ as a groupoid consisting of identities) is a subalgebra of $\Bbbk \mathscr G$.  Also observe that $\Bbbk \mathscr G\skel 0$ is just the ring of $\Bbbk$-valued continuous functions with compact support on $\mathscr G\skel 0$ with pointwise multiplication, and hence is commutative.

The algebra $\mathbb C\mathscr G$, in the case where $\mathscr G$ is Hausdorff, has been further studied in~\cite{operatorguys1,operatorsimple1}; see also~\cite{operatorsimple2}.  Morita equivalence  with applications to Leavitt path algebras~\cite{Leavittsemiprimitivity} was studied in~\cite{GroupoidMorita} and a sheaf representation of $\Bbbk\mathscr G$-modules was obtained in~\cite{groupoidbundles}. Recently, the growth of \'etale groupoid algebras was studied in~\cite{Nekgrpd}.

The algebra $\Bbbk\mathscr G$ admits an involution $\p\mapsto \check{\p}$ where $\check{\p}(g)=\p(g\inv)$ and so dual notions like right and left primitivity are equivalent for $\Bbbk\mathscr G$.  Hence we shall work alternatively with right and left modules as we find convenient.  The algebra $\Bbbk\mathscr G$ is unital if and only if $\mathscr G\skel 0$ is compact, but it always has local units~\cite{mygroupoidalgebra,groupoidbundles}.

\section{Simplicity of groupoid algebras}
The results of this section were first obtained in~\cite{operatorsimple1} for Hausdorff groupoids over $\Bbbk=\mathbb C$ (and the norm on $\mathbb C$ was used in the proof). Here we give proofs that avoid using the norm, drop the Hausdorff condition whenever possible and consider more ground general rings.  Our techniques are different as well, in that we use the Sch\"utzenberger representations of~\cite{mygroupoidalgebra}. These results were announced by the author at the workshop ``Semigroups and Applications" (Uppsala, August 2012), but never previously published.  Similar results were obtained independently in the meantime in~\cite{operatorsimple2}.  Nonetheless, we produce here a complete proof of the simplicity criterion as we shall need some of the intermediary results in the sequel.

\begin{Lemma}\label{cutdowntochar}
Let $\mathscr G$ be an ample groupoid and $\Bbbk$ a commutative ring with unit.  Suppose that $0\neq \varphi\in \Bbbk \mathscr G\skel 0$ and $I$ is an ideal of $\Bbbk \mathscr G$ containing $\varphi$.  Then $I$ contains a non-zero element of the form $k\cdot \chi_U$ where $k\in \Bbbk$ and $U\subseteq \mathscr G\skel 0$ is compact open.
\end{Lemma}
\begin{proof}
Let $k\in \Bbbk$ be a non-zero element in the image of $\p$.  Then $U=\p\inv (k)$ is a non-empty compact open subset of $\mathscr G\skel 0$ and $\p\ast \chi_U = k\cdot \chi_U$ is in $I$.
\end{proof}

The following lemma (and its dual) will also be useful.
\begin{Lemma}\label{hitright}
Let $\p\in \Bbbk\mathscr G$ and $U\in \Bis_c(\mathscr G)$.  Suppose that $h\in U$ and $\dom(g)=\ran(h)$.  Then \[\p\ast \chi_U(gh)=\p(g)\] and so in particular is non-zero if $\p(g)\neq 0$.
\end{Lemma}
\begin{proof}
The unique element $y\in U$ with $\dom(y)=\dom(gh)$ is $h$ and so \[\p\ast \chi_U(gh)=\sum_{\dom(y)=\dom(gh)}\p(ghy\inv)\chi_U(y)=\p(g)\] as required.
\end{proof}

Now we prove an analogue of the Cuntz-Krieger uniqueness theorem in the context of Hausdorff ample groupoids.  The proof is based on the idea of~\cite{operatorsimple1}, but we avoid using the norm.  A similar result is embedded in the proof of the main result of~\cite{operatorsimple2}.

\begin{Prop}\label{effectivecase}
Let $\mathscr G$ be an effective Hausdorff ample groupoid and $\Bbbk$ a commutative ring with unit.  Suppose that $I$ is a non-zero ideal of $\Bbbk \mathscr G$.  Then $I$ contains a non-zero element of the form $k\cdot \chi_U$ with $k\in \Bbbk\setminus \{0\}$ and $U\subseteq \mathscr G\skel 0$ a non-empty compact open subset.
\end{Prop}
\begin{proof}
Let $\p\in I\setminus \{0\}$ and suppose that $\p(g)\neq 0$ with $g\in \mathscr G\skel 1$.  Let $U\in \Bis_c(\mathscr G)$ contain $g\inv$ and put $x=\ran(g)$. Then $\p\ast \chi_U(x) =\p(g)\neq 0$ by Lemma~\ref{hitright}.  Thus there is an element $\psi\in I$ with $\psi|_{\mathscr G\skel 0}\neq 0$ (take $\psi=\p\ast \chi_U$ as above).  As $\mathscr G\skel 0$ is clopen in $\mathscr G\skel 1$ by the Hausdorff property, we have that $\psi|_{\mathscr G\skel 0}\in \Bbbk \mathscr G\skel 0$.  Let $K$ be the support of $\Psi=\psi-\psi|_{\mathscr G\skel 0}$.  Then $K\subseteq \mathscr G\skel 1\setminus \mathscr G\skel 0$ is compact open.

Write $\psi|_{\mathscr G\skel 0} = \sum_{i=1}^m k_i\cdot \chi_{U_i}$ with the $U_i$ disjoint compact open subsets of $\mathscr G\skel 0$ and the $k_i\in \Bbbk\setminus \{0\}$.  By~\cite[Lemma~3.1]{operatorsimple1}, the effectiveness of $\mathscr G$ implies that there is a non-empty open subset $V\subseteq U_1$ with $VKV=
\emptyset$.  Since $\mathscr G\skel 0$ has a basis of compact open sets, we may assume that $V$ is compact.  We claim that $k_1\cdot \chi_V=\chi_V\ast \psi\ast \chi_V\in I$.  Indeed, Lemma~\ref{hitright} and its dual imply $\chi_V\ast \Psi\ast \chi_V=0$ because $VKV=\emptyset$.  Thus $\chi_V\ast \psi\ast \chi_V= \chi_V\ast \psi|_{\mathscr G\skel 0}\ast \chi_V=k_1\cdot \chi_V$ is in $I$, as required.
\end{proof}

Our next result removes the Hausdorff condition from a result of~\cite{operatorsimple1} and works over any base ring.  See also~\cite{operatorsimple2} where more or less the same result is obtained in a slightly different way.

Let $\mathscr G$ be an ample groupoid and $\Bbbk$ a commutative ring with unit.
If $x\in \mathscr G\skel 0$, we put $L_x=\dom^{-1}(x)$.  There is a $\Bbbk \mathscr G$-module structure on $\Bbbk L_x$ given by \[\p\cdot t=\sum_{y\in L_x}\p(yt\inv)y=\sum_{d(g)=r(t)}\p(g)gt\] for $t\in L_x$.  Moreover, if $U\in \Bis_c(\mathscr G)$,
\begin{equation*}\label{schutzformula}
\chi_U\cdot t=\begin{cases}ht, & \text{if}\ h\in U,\ \dom(h)=\ran(t)\\ 0, & \text{if}\ \ran(t)\notin \dom(U)=U^*U\end{cases}
\end{equation*}
for $t\in L_x$.
Note that if $\ran(t)\in U^*U$, then $h$ above is unique.  Also the isotropy group $G_x$ of $\mathscr G$ at $x$ acts freely on the right of $L_x$ and $\Bbbk L_x$ is a $\Bbbk \mathscr G$-$\Bbbk G_x$-bimodule.  See~\cite[Proposition~7.8]{mygroupoidalgebra} for details.  If $V$ is a $\Bbbk G_x$-module, then $\Ind_x(V)=\Bbbk L_x\otimes_{\Bbbk G_x} V$ is a $\Bbbk\mathscr G$-module.  Moreover, the functor $\Ind_x$ is exact and sends (semi)simple modules to (semi)simple modules (see~\cite[Proposition~7.19]{mygroupoidalgebra} for the latter property).

The annihilator ideal of $\Bbbk L_x$ is a proper ideal because if $U$ is a compact open neighborhood of $x$ in $\mathscr G\skel 0$, then $\chi_U\cdot x=x$.  If we give $\Bbbk$ the trivial $\Bbbk G_x$-module structure, then $\Ind_x(\Bbbk)\cong \Bbbk \mathcal O_x$ with the action given by \[\p\cdot u=\sum_{\dom(g)=u}\p(g)\ran(g)\] for $u\in \mathcal O_y$. In particular, if $U\in \Bis_c(\mathscr G)$, then \[\chi_U\cdot u=\begin{cases} U\cdot u, & \text{if}\ u\in U^*U\\ 0, & \text{else.}\end{cases}\] If $x\in U$ with $U\in \Bis_c(\mathscr G\skel 0)$, then $\chi_U\cdot x=x$ and so the annihilator of $\Bbbk O_x$ is proper.

\begin{Prop}\label{minimal}
Let $\mathscr G$ be an ample groupoid and $\Bbbk$ a commutative ring with unit.  Then $\mathscr G$ is minimal if and only if $\chi_U$ generates $\Bbbk \mathscr G$ as an ideal for all $U\subseteq \mathscr G\skel 0$ a non-empty compact open subset.
\end{Prop}
\begin{proof}
Suppose first that $\chi_U$ generates $\Bbbk \mathscr G$ as an ideal for all non-empty compact open subsets $U\in \mathscr G\skel 0$.  Let $x\in \mathscr G\skel 0$.  The annihilator of $\Bbbk L_x$ is a proper ideal in $\Bbbk\mathscr G$ and therefore $\chi_U$ does not annihilate $\Bbbk L_x$.  Thus there exists $t\in L_x$ with $\chi_U\cdot t\neq 0$.  This means that $\ran(t)\in U$ and so $U\cap \mathcal O_x\neq \emptyset$.  We conclude that all orbits are dense and hence $\mathscr G$ is minimal.

Next suppose that $\mathscr G$ is minimal and let $I$ be the ideal generated by $\chi_U$.  Let $S=\{V\in \Bis_c(\mathscr G)\mid \chi_V\in I\}$.  Then $S$ is an ideal of $\Bis_c(\mathscr G)$ and $E(S)$ is a generalized boolean algebra (since $\chi_{V_1\cup V_2}=\chi_{V_1}+\chi_{V_2}-\chi_{V_1}\ast \chi_{V_2}$ and $\chi_{V_1\setminus V_2}=\chi_{V_1}-\chi_{V_1}\ast\chi_{V_2}$ for $V_1,V_2\in \Bis_c(\mathscr G\skel 0)$).  We claim that $E(S)=\Bis_c(\mathscr G\skel 0)$. Indeed, if $x\in \mathscr G\skel 0$, then since $\mathscr G$ is minimal we have that $U\cap \mathcal O_x$ is non-empty.  Suppose that $y\in U\cap \mathcal O_x$ and $g\in \mathscr G\skel 1$ with $\ran(g)=y$ and $\dom(g)=x$.  Let $V\in \Bis_c(\mathscr G)$ with $g\in V$.  Then $x\in (UV)^*UV\in E(S)$.  Thus $E(S)$ contains a compact open neighborhood of each point of $\mathscr G\skel 0$.  Also $E(S)$ is an order ideal in $\Bis_c(\mathscr G\skel 0)$.  Thus $E(S)$ contains a basis of compact open subsets of $\mathscr G\skel 0$.  But then since $E(S)$ is closed under finite unions, we conclude $E(S)=\Bis_c(\mathscr G\skel 0)$.  But no proper ideal in an inverse semigroup can contain all the idempotents and so $S=\Bis_c(\mathscr G)$, whence $I=\Bbbk \mathscr G$.
\end{proof}

The following theorem generalizes the results of~\cite{operatorsimple1} and was obtained in~\cite{operatorsimple2} using a similar method.

\begin{Thm}\label{simple}
Let $\mathscr G$ be an ample groupoid and $\Bbbk$ a field.  If $\Bbbk \mathscr G$ is simple, then $\mathscr G$ is effective and minimal.  The converse holds if $\mathscr G$ is Hausdorff.
\end{Thm}
\begin{proof}
Suppose first that $\Bbbk \mathscr G$ is simple. Minimality is immediate from Proposition~\ref{minimal}. To show that $\mathscr G$ is effective, it suffices to prove that if $U\in \Bis_c(\mathscr G)$ is contained in $\Is(\mathscr G)$, then it is contained in $\mathscr G\skel 0$. Of course, we may assume that $U$ is non-empty. Because $U$ is contained in the isotropy subgroupoid, it follows that $\chi_U-\chi_{U^*U}$ annihilates each of the modules $\Bbbk \mathcal O_x$ with $x\in \mathscr G\skel 0$.  But these modules have proper annihilator ideals and hence are faithful by simplicity of $\Bbbk\mathscr G$.  We conclude that $U=U^*U$ and hence $U\subseteq \mathscr G\skel 0$.

Assume next that $\mathscr G$ is effective, minimal and  Hausdorff.  If $I$ is a non-zero ideal of $\Bbbk \mathscr G$, then by Proposition~\ref{effectivecase} it contains an element of the form $\chi_U$ with $\emptyset\neq U\in \Bis_c(\mathscr G\skel 0)$. We conclude from Proposition~\ref{minimal} that $I=\Bbbk \mathscr G$.  This completes the proof.
\end{proof}

We leave it as an open question as to whether the Hausdorff condition is really needed in the converse (we suspect that it is).  Basically it boils down to whether Proposition~\ref{effectivecase} is true when $\mathscr G$ is not Hausdorff.

\begin{Cor}\label{simplehaus}
Let $\mathscr G$ be a Hausdorff ample groupoid and $\Bbbk$ a field.  Then $\Bbbk \mathscr G$ is simple if and only if $\mathscr G$ is effective and minimal.
\end{Cor}

We end this section with two further results on effective Hausdorff ample groupoids.
The first result, characterizing the center of an effective Hausdorff groupoid, generalizes a result of~\cite{operatorsimple2}.

\begin{Prop}
Let $\mathscr G$ be an effective Hausdorff ample groupoid and $\Bbbk$ a commutative ring with unit.  Then the center $Z(\Bbbk \mathscr G)$ of $\Bbbk \mathscr G$ consists of those continuous functions with compact support $\p\colon \mathscr G\skel 0\to \Bbbk$ which are constant on orbits.  In particular, if $\mathscr G\skel 0$ has a dense orbit, then the equality \[Z(\Bbbk \mathscr G)=\begin{cases} \Bbbk\cdot \chi_{\mathscr G\skel 0}, & \text{if}\ \mathscr G\skel 0\ \text{is compact}\\ 0, & \text{else}\end{cases}\] holds.
\end{Prop}
\begin{proof}
By~\cite[Proposition~4.13]{mygroupoidalgebra} $\p\in Z(\Bbbk \mathscr G)$ if and only if $\p$ is supported on $\Is(\mathscr G)$ and $\p(gzg\inv)=\p(z)$ whenever $\dom(g)=\ran(z)=\dom(z)$.  Since the support of $\p$ is open (because $\mathscr G$ is Hausdorff), we in fact have that the support of $\p$ is contained in $\mathscr G\skel 0$ because $\mathscr G$ is effective.  It is now immediate that $Z(\Bbbk \mathscr G)$ consists of those $\p\in \Bbbk\mathscr G\skel 0$ which are constant on orbits.

The final statement is clear since if $\mathcal O$ is a dense orbit and $\p\in Z(\mathscr G)$, then $\p$ is constant on $\mathcal O$ and hence on $\ov {\mathcal O}=\mathscr G\skel 0$.
\end{proof}

We record here another result for effective Hausdorff groupoids.

\begin{Prop}
Let $\mathscr G$ be an effective Hausdorff ample groupoid and $\Bbbk$ a commutative ring with unit.  Then $\Bbbk \mathscr G\skel 0$ is a maximal commutative subalgebra of $\Bbbk\mathscr G$.
\end{Prop}
\begin{proof}
Clearly $\Bbbk \mathscr G\skel 0$ is a commutative subalgebra.  Suppose that $0\neq \p\in \Bbbk\mathscr G$ centralizes $\Bbbk \mathscr G\skel 0$.  Because $\mathscr G$ is Hausdorff, $\p\inv(\Bbbk\setminus \{0\})$ is compact open.  By effectiveness, it suffices to show that the support of $\p$ is contained in $\Is(\mathscr G)$. Suppose $\p(g)\neq 0$ with $\dom(g)\neq \ran(g)$.  Then there exists $U\in \Bis_c(\mathscr G\skel 0)$ such that $\dom(g)\in U$ and $\ran(g)\notin U$.  But then $\p\ast \chi_U(g)=\p(g)\neq 0$ and $\chi_U\ast \p(g)=0$, a contradiction.  This completes the proof.
\end{proof}

\section{Primitivity and semiprimitivity}

Recall that a ring is \emph{primitive} if it has a faithful simple module and \emph{semiprimitive} if it has a faithful semisimple module (cf.~\cite{LamBook}).  We investigate primitivity and semiprimitivity of ample groupoid algebras.  This will allow us to recover results of Domanov~\cite{Domanov} and Munn~\cite{MunnSemiprim,MunnSemiprim2,MunnAlgebraSurvey} for inverse semigroup algebras (with easier proofs) and the semiprimitivity of Leavitt path algebras~\cite{Leavittsemiprimitivity}. Also an extension of the result of Amitsur~\cite{Amitsur} that a group algebra over a field of characteristic $0$, which is not algebraic over $\mathbb Q$, is semiprimitive is obtained for ample groupoid algebras.

\subsection{Semiprimitivity}
We first establish that effective Hausdorff groupoids always have semiprimitive algebras over a semiprimitive base ring.

\begin{Prop}\label{semiprimitivityforeffective}
Let $\mathscr G$ be a Hausdorff effective ample groupoid and $\Bbbk$ a commutative ring with unit.  Then $\Bbbk \mathscr G$ is semiprimitive whenever $\Bbbk$ is semiprimitive.
\end{Prop}
\begin{proof}
Assume $\Bbbk$ is semiprimitive and let $V$ be a faithful semisimple $\Bbbk$-module.  Consider the $\Bbbk \mathscr G$-module $M=\bigoplus_{x\in \mathscr G\skel 0} \Ind_x(V)$ where we view $V$ as a $\Bbbk G_x$-module via the trivial action of $G_x$.  Then $M$ is a semisimple module by~\cite[Proposition~7.19]{mygroupoidalgebra}.   Suppose that $M$ is not faithful.  Then its annihilator contains an element of the form $k\cdot \chi_U$ with $k\in \Bbbk\setminus \{0\}$ and $U\in \Bis_c(\mathscr G\skel 0)$ non-empty by Proposition~\ref{effectivecase}.  Let $v\in V$ with $kv\neq 0$ and let $x\in U$.  Then $k\cdot \chi_U(x\otimes v) = x\otimes kv\neq 0$ as $\Ind_x(V)= \bigoplus_{y\in \mathcal O_x}y\otimes V$ as a $\Bbbk$-module.  This contradiction shows that $M$ is faithful.
%
\end{proof}

Let $\mathscr G$ be an ample groupoid and $\Bbbk$ a commutative ring with unit.  Let us say that $X\subseteq \mathscr G\skel 0$ is \emph{$\Bbbk$-dense} if, for all $0\neq \p\in \Bbbk \mathscr G$, there exists $g\in \mathscr G\skel 1$ such that:
\begin{enumerate}
\item $\dom(g)\in X$;
\item $\p(g)\neq 0$.
\end{enumerate}
Notice that if $X$ is $\Bbbk$-dense, then so is each $Y\supseteq X$. Also note that the condition $\dom(g)\in X$ could be replaced with $\ran(g)\in X$ by considering $\check{\p}$.

The following proposition justifies the terminology.

\begin{Prop}\label{kdense}
Let $\mathscr G$ be an ample groupoid and $X\subseteq \mathscr G\skel 0$.
\begin{enumerate}
\item If $X$ is $\Bbbk$-dense, then $X$ is dense in $\mathscr G\skel 0$.
\item If $\mathscr G$ is Hausdorff and $X$ is dense in $\mathscr G\skel 0$, then $X$ is $\Bbbk$-dense.
\end{enumerate}
\end{Prop}
\begin{proof}
Assume that $X$ is $\Bbbk$-dense and $U\subseteq \mathscr G\skel 0$ is compact open.  Then $\chi_U(x)\neq 0$ for some $x\in X$ and hence $U\cap X\neq \emptyset$.  Thus $X$ is dense in $\mathscr G\skel 0$.

Suppose that $\mathscr G$ is Hausdorff and $X$ is dense.  Let $\p\in \Bbbk\mathscr G$.  Then $V=\p\inv (\Bbbk\setminus \{0\})$ is compact open and hence $\dom(V)$ is open.  Thus there exists $x\in X\cap \dom(V)$, that is, there exists $g\in \mathscr G\skel 1$ with $\p(g)\neq 0$ and $\dom(g)\in X$.
\end{proof}

Example~\ref{examplecliff} below shows that $\Bbbk$-density can be different than density for non-Hausdorff groupoids.

In order to prove our main semiprimitivity result we need a condition under which a module induced from an isotropy group is not annihilated by some element of $\Bbbk\mathscr G$.

\begin{Lemma}\label{l:induceup}
Let $\Bbbk$ be a commutative ring with unit and $\mathscr G$ an ample groupoid.
Let $x\in \mathscr G\skel 0$ and let $V_x$ be a faithful $\Bbbk G_x$-module.  Suppose that $\p\in \Bbbk\mathscr G$ satisfies $\p(g)\neq 0$ for some $g\in \mathscr G\skel 1$ with $\dom(g)\in \mathcal O_x$.  Then $\p\cdot \Ind_x(V_x)\neq 0$.
\end{Lemma}
\begin{proof}
For each $y$ in the orbit $\mathcal O_x$ of $x$, choose an arrow $h_y\colon x\to y$.
Set
\[a=\sum_{\dom (z)=\dom(g),\ran(z)=\ran(g)}\p(z)(h_{\ran(g)}\inv zh_{\dom(g)})\in \Bbbk G_x.\] Note that $g$ is the unique element $z$ with $\dom(z)=\dom(g)$, $\ran(z)=\ran(g)$ and $h_{\ran(g)}\inv zh_{\dom(g)}=h_{\ran(g)}\inv gh_{\dom(g)}$ and so the coefficient of $h_{\ran(g)}\inv gh_{\dom(g)}$ in $a$ is $\p(g)\neq 0$, whence $a\neq 0$.  Because $V_x$ is a faithful $\Bbbk G_x$-module, we can choose $v\in V_x$ with $av\neq 0$.

Recall that \[\Ind_x(V_x)=\Bbbk L_x\otimes_{\Bbbk G_x} V_x=\bigoplus_{y\in \mathcal O_x}h_y\otimes V_x\] where the direct sum decomposition is as a $\Bbbk$-module~\cite{mygroupoidalgebra}.    Now we compute
\begin{align*}
\p\cdot (h_{\dom(g)}\otimes v) &= \sum_{\dom(z)=\dom(g)} \p(z)zh_{\dom(g)}\otimes v \\ &= \sum_{\dom(z)= \dom(g)}\p(z)h_{\ran(z)}(h_{\ran(z)}\inv zh_{\dom(g)})\otimes v \\ &=  \sum_{\dom(z)=\dom(g)}h_{\ran(z)}\otimes \p(z)(h_{\ran(z)}\inv zh_{\dom(z)})v \\ &= \sum_{y\in \mathcal O_x}h_y\otimes \sum_{\dom(z)=\dom(g),\ran(z)=y}\p(z)(h_y\inv zh_{\dom(g)})v \\
&= h_{\ran(g)}\otimes av+ {}\\ & \qquad\quad \sum_{y\in \mathcal O_x\setminus \{\ran(g)\}}h_y\otimes \sum_{\dom(z)=\dom(g),\ran(z)=y}\p(z)(h_y\inv zh_{\dom(g)})v\\ & \neq 0
\end{align*}
by choice of $v$.
\end{proof}

We now prove that if the algebras of sufficiently many isotropy groups are semiprimitive over $\Bbbk$, then so is the groupoid algebra.  This is a generalization of a result of Domanov~\cite{Domanov} for inverse semigroups.

\begin{Thm}\label{t:semiprimitivity}
Let $\mathscr G$ be an ample groupoid and $\Bbbk$ a commutative ring with unit. Suppose that the set $X$ of $x\in \mathscr G\skel 0$ such that $\Bbbk G_x$ is semiprimitive is $\Bbbk$-dense.  Then $\Bbbk \mathscr G$ is semiprimitive.
\end{Thm}
\begin{proof}
Choose a faithful semisimple module $V_x$ for $\Bbbk G_x$ for each $x\in X$ and put $V=\bigoplus_{x\in X} \Ind_x(V_x)$.  Then $V$ is semisimple by~\cite[Proposition~7.19]{mygroupoidalgebra}.  We need to show that it is faithful.  Let $0\neq \p\in \Bbbk \mathscr G$ and let $g\in \mathscr G\skel 1$ with $\p(g)\neq 0$ and $x=\dom(g)\in X$. Then $\p\cdot \Ind_x(V_x)\neq 0$ by Lemma~\ref{l:induceup} and hence $\p\cdot V\neq 0$. This proves that $V$ is faithful and hence $\Bbbk\mathscr G$ is semiprimitive.
\end{proof}

The groupoid used in~\cite{GroupoidMorita} to realize Leavitt path algebras has the property that each isotropy group is either trivial or infinite cyclic.  Hence one has that the isotropy groups have semiprimitive algebras over any base field.  We thus recover the following result of~\cite{Leavittsemiprimitivity}.

\begin{Cor}
Leavitt path algebras are semiprimitive over any base field.
\end{Cor}

The following extends a celebrated result of Amitsur~\cite{Amitsur} to ample groupoid algebras.

\begin{Cor}
Let $\mathscr G$ be an ample groupoid and let $\Bbbk$ be a field of characteristic $0$ that is not algebraic over $\mathbb Q$.   Then $\Bbbk\mathscr G$ is semiprimitive.
\end{Cor}
\begin{proof}
By a well-known result of Amitsur~\cite{Amitsur}, $\Bbbk G$ is semiprimitive for any group $G$. Since $\mathscr G\skel 0$ is obviously $\Bbbk$-dense, we conclude that $\Bbbk \mathscr G$ is semiprimitive by Theorem~\ref{t:semiprimitivity}.
\end{proof}

In Example~\ref{e:needkdense} below, we shall show that $\Bbbk$-density cannot be relaxed to just density in Theorem~\ref{t:semiprimitivity}.

Our next goal is to generalize a result of Munn~\cite{MunnSemiprim} from inverse semigroups to groupoids.  In fact, his result is trivial in the groupoid context.

\begin{Prop}\label{p:isolated}
Let $\mathscr G$ be an ample groupoid and $\Bbbk$ a commutative ring with unit.  Let $U\subseteq \mathscr G\skel 0$ be compact open and let $\mathscr G|_U$ be the full open subgroupoid with object set $U$.  If $\Bbbk\mathscr G$ is (semi)primitive, then so $\Bbbk \mathscr G|_U$.  In particular, if $x\in \mathscr G_0$ is an isolated point, then $\Bbbk G_x$ is (semi)primitive whenever $\Bbbk\mathscr G$ is (semi)primitive.
\end{Prop}
\begin{proof}
It is shown in~\cite{groupoidbundles} that $\Bbbk\mathscr G|_U$ is the corner $\chi_U\cdot \Bbbk\mathscr G\cdot \chi_U$.  As a corner in a (semi)primitive ring is (semi)primitive (cf.~\cite[Corollary~21.13]{LamBook}), the result follows.
\end{proof}

Combining Proposition~\ref{p:isolated} with Theorem~\ref{t:semiprimitivity}, we obtain the following corollary.

\begin{Cor}\label{c:isolatedsemi}
Let $\mathscr G$ be an ample groupoid and $\Bbbk$ a commutative ring with unit.  Suppose that $X\subseteq \mathscr G\skel 0$ is a $\Bbbk$-dense set of isolated points. Then $\Bbbk\mathscr G$ is semiprimitive if and only if $\Bbbk G_x$ is semiprimitive for all $x\in X$.
\end{Cor}

\subsection{Primitivity}
Now we consider primitivity of ample groupoid algebras.
First we show that if $\Bbbk\mathscr G$ is primitive, then $\mathscr G$ admits a $\Bbbk$-dense orbit (and hence a dense orbit).

Let $\mathscr G$ be an \'etale groupoid.  A \emph{$\mathscr G$-sheaf} consists of a space $E$, a local homeomorphism $p\colon E\to \mathscr G\skel 0$ and an action map $E\times_{\mathscr G\skel 0} \mathscr G\skel 1\to E$ (where the fiber product is with respect to $p$ and $\ran$), denoted $(x,g)\mapsto xg$ satisfying the following axioms:
\begin{itemize}
\item $ep(e)=e$ for all $e\in E$;
\item $p(eg)=\dom(g)$ whenever $p(e)=\ran(g)$;
\item $(eg)h=e(gh)$ whenever $p(e)=\ran(g)$ and $\dom(g)=\ran(h)$.
\end{itemize}

 If $\Bbbk$ is a commutative ring with unit, then a \emph{$\mathscr G$-sheaf of $\Bbbk$-modules} is a $\mathscr G$-sheaf $(E,p)$ together a $\Bbbk$-module structure on each stalk $E_x=p\inv(x)$ such that:
\begin{itemize}
\item the zero section, denoted $0$, sending $x\in \mathscr G\skel 0$ to the zero of $E_x$ is continuous;
\item addition $E\times_{\mathscr G\skel 0} E\to E$ is continuous;
\item scalar multiplication $\Bbbk\times E\to E$ is continuous;
\item for each $g\in \mathscr G\skel 1$, the map $R_g\colon E_{\ran(g)}\to E_{\dom(x)}$ given by $R_g(e) = eg$ is $\Bbbk$-linear;
\end{itemize}
where $\Bbbk$ has the discrete topology in the third item.
Note that the first three conditions are equivalent to $(E,p)$ being a sheaf of $\Bbbk$-modules over $\mathscr G\skel 0$.

If $(E,p)$ is a $\mathscr G$-sheaf of $\Bbbk$-modules, then $\Gamma_c(E,p)$ denotes the $\Bbbk$-module of global sections of $p$ with compact support.  There is a right $\Bbbk\mathscr G$-module structure on $\Gamma_c(E,p)$ given by  \[(s\p)(x) = \sum_{\dom(g)=x} \p(g)s(\ran(g))g=\sum_{\dom(g)=x} \p(g)R_g(s(\ran(g))). \]  It is shown in~\cite{groupoidbundles} that every unitary right $\Bbbk\mathscr G$-module is isomorphic to $\Gamma_c(E,p)$ for some $\mathscr G$-sheaf $(E,p)$ of $\Bbbk$-modules.  We recall that $M$ is unitary if $M\cdot \Bbbk\mathscr G=M$.

\begin{Prop}\label{p:denseorbit}
Let $\mathscr G$ be an ample groupoid and $\Bbbk$ a field.  Suppose that $\Bbbk\mathscr G$ is primitive.  Then $\mathscr G$ has a $\Bbbk$-dense orbit.
\end{Prop}
\begin{proof}
Let $M$ be a faithful simple right $\Bbbk\mathscr G$-module.  Then $M$ is unitary and hence $M\cong \Gamma_c(E,p)$ for some $\mathscr G$-sheaf of $\Bbbk$-vector spaces $(E,p)$~\cite{groupoidbundles}.  Suppose that the stalk $E_x$ is non-zero.  We claim that the orbit $\mathcal O_x$ is $\Bbbk$-dense. Indeed, let $X=\ran\inv(\mathcal O_x)$ and let  $I$ be the set of all $\p\in \Bbbk\mathscr G$ that vanish on $X$.  Then $I$ is a right ideal. If $I=0$, then $\mathcal O_x$ is $\Bbbk$-dense.  So suppose that $I\neq 0$.  Because $M$ is faithful, we can find $t\in \Gamma_c(E,p)$ with $tI\neq 0$.  Then $tI=\Gamma_c(E,p)$ by simplicity.  So if $s\in \Gamma_c(E,p)$, then $s=t\p$ for some $\p\in I$.  But then, \[s(x)=(t\p)(x)=\sum_{\dom(g)=x}\p(g)t(\ran(g))g=0\] because $\p$ vanishes on $X$ and $d(g)=x$ implies $g\in X$.  This contradicts $E_x\neq 0$. We conclude that $\mathcal O_x$ is $\Bbbk$-dense.
\end{proof}

We now give several situations under which the converse holds.

\begin{Thm}
Let $\mathscr G$ be an effective Hausdorff ample groupoid and $\Bbbk$ a field.  Then $\Bbbk\mathscr G$ is primitive if and only if $\mathscr G$ has  a dense orbit.
\end{Thm}
\begin{proof}
It is necessary for $\mathscr G$ to have a dense orbit by Proposition~\ref{p:denseorbit} and Proposition~\ref{kdense}.  For sufficiency, assume that $\mathscr G\skel 0$ has a dense orbit $\mathcal O_x$.  We claim that $\Bbbk \mathcal O_x=\Ind_x(\Bbbk)$ is a faithful simple module (where $\Bbbk$ has the trivial $G_x$-action).  We know that it is simple by~\cite[Proposition~7.19]{mygroupoidalgebra}.  If the annihilator is non-zero, then by Proposition~\ref{effectivecase} it contains an element of the form $\chi_U$ with $U\in \Bis_c(\mathscr G\skel 0)$.  By density, there exists $y\in \mathcal O_x\cap U$.  Then $\chi_Uy=y\neq 0$.  This contradiction shows that $\Bbbk\mathscr G$ is primitive.
\end{proof}

Next we show that if there is a $\Bbbk$-dense orbit whose isotropy group has a primitive algebra, then the groupoid has a primitive algebra.

\begin{Thm}\label{t:primitivity}
Let $\mathscr G$ be an ample groupoid and $\Bbbk$ a field.  Suppose that $x\in \mathscr G\skel 0$ is such that $\mathcal O_x$ is $\Bbbk$-dense. If $\Bbbk G_x$ is primitive, then $\Bbbk\mathscr G$ is primitive.  The converse holds if $x$ is an isolated point.
\end{Thm}
\begin{proof}
The final statement follows from Proposition~\ref{p:isolated}.  Suppose that $\Bbbk G_x$ is primitive and that $V_x$ is a faithful simple $\Bbbk G_x$-module.  Then $\Ind_x(V_x)$ is simple by~\cite[Proposition~7.19]{mygroupoidalgebra} and faithful by Lemma~\ref{l:induceup} and the definition of $\Bbbk$-density.
\end{proof}

Inverse semigroups with zero are constructed in~\cite{Munntwoexamples} with the property that their associated ample groupoids have simple (hence primitive) algebras over $\mathbb F_p$ but with the property that none of their isotropy groups have a semiprimitive algebra over $\mathbb F_p$ (in fact, they contain non-zero nilpotent ideals).

\section{Applications to inverse semigroups}
In this section, we apply the results of the previous sections to obtain new results about inverse semigroup algebras, as well as simpler and more conceptual proofs of old results.

\subsection{The universal groupoid}
First we recall the construction of the universal groupoid $\mathscr U(S)$ of an inverse semigroup and the contracted universal groupoid $\mathscr U_0(S)$ for an inverse semigroup with zero. See~\cite{Exel,Paterson,mygroupoidalgebra} for details.

A \emph{character} of a semilattice $E$ is a non-zero homomorphism $\theta\colon E\to \{0,1\}$ where $\{0,1\}$ is a semilattice under multiplication.  The \emph{spectrum} of $E$ is the space $\wh E$ of characters of $E$, topologized as a subspace of $\{0,1\}^E$.  Note that $\wh E$ is Hausdorff with a basis of compact open sets.  Indeed, if we put $D(e)=\{\theta\in \wh E\mid \theta(e)=1\}$ for $e\in E(S)$, then the sets of the form $D(e)\cap D(e_1)^c\cap\cdots D(e_n)^c$ form a basis of compact open sets for the topology, where $X^c$ denotes the complement of $X$.   If $e\in E$, then the \emph{principal character} $\chi_e\colon E\to \{0,1\}$ is defined by
\[\chi_e(f)=\begin{cases} 1, & \text{if}\ f\geq e\\ 0, & \text{else.}\end{cases}\]  The principal characters are dense in $\wh E$.  If $E$ has a zero element, then a character $\theta$ is called \emph{proper} if $\theta(0)=0$, or equivalently $\theta\neq \chi_0$.  The set of proper characters will be denoted $\wh E_0$.  Notice that $D(0)=\{\theta_0\}$ and so $\theta_0$ is always an isolated point.

Let $S$ be an inverse semigroup. Then $S$ acts on $\wh{E(S)}$.  The domain of the action of $s$ is $D(s^*s)$.  If $\theta\in D(s^*s)$, then $(s\theta)(e) = \theta(s^*es)$.  If $S$ has a zero, then $\wh{E(S)}_0$ is invariant under $S$.  The \emph{universal groupoid} of $S$ is $\mathscr U(S)=S\ltimes \wh{E(S)}$. Note that $[s,\chi_{s^*s}]\in (t,D(t^*t))$ if and only if $s\leq t$ and that the isotropy group $G_{\chi_e}$ of a principal character is isomorphic to the maximal subgroup $G_e$ (cf.~\cite{mygroupoidalgebra}).
It is known that $\mathscr U(S)$ is Hausdorff if and only if $S$ is Hausdorff~\cite[Theorem~5.17]{mygroupoidalgebra}.

If $S$ has a zero, we put $\mathscr U_0(S)=S\ltimes \wh{E(S)}_0$ and call it the \emph{contracted universal groupoid} of $S$.

The following theorem is fundamental to the subject.

\begin{Thm}\label{t:isothm}
Let $S$ be an inverse semigroup and $\Bbbk$ a commutative ring with unit.  Then $\Bbbk S\cong \Bbbk\mathscr U(S)$. The isomorphism sends $s\in S$ to $\chi_{(s,D(s^*s))}$.  If $S$ has a zero, then $\Bbbk_0S\cong \Bbbk\mathscr U_0(S)$.
\end{Thm}
\begin{proof}
The first isomorphism is proved in~\cite[Theorem~6.3]{mygroupoidalgebra}. For the second isomorphism, note that $0$ is sent to $\chi_{(0,D(0))}$ and so $\Bbbk_0S\cong \Bbbk\mathscr U(S)/\Bbbk \chi_{(0,D(0))}$.  But $(0,D(0))$ just consists of the unit $\chi_0$ and hence $\Bbbk \chi_{(0,D(0))}$ is the kernel of the restriction map $\Bbbk\mathscr U(S)\to \Bbbk\mathscr U(S)_0$, and the latter is a surjective homomorphism because $\mathscr U_0(S)$ is a clopen full subgroupoid of $\mathscr U(S)$ whose unit space is a union of orbits.
\end{proof}

\subsubsection{Ultrafilters and tight characters}
Let $E$ be a semilattice.  A \emph{filter} is a non-empty subsemigroup $\mathcal F\subseteq E$ with the property that $e\geq f$ with $f\in \mathcal F$ implies $e\in \mathcal F$.  The characters of $E$ are exactly the characteristic functions $\chi_{\mathcal F}$ of filters $\mathcal F$.  Let $E$ be a semilattice with zero. The proper characters of $E$ are in bijection with proper filters.  A maximal proper filter is called an \emph{ultrafilter}.  Denote by $UF(E)$ the subspace of $\wh E_0$ consisting of those $\theta\in \wh E$ with $\theta\inv (1)$ an ultrafilter.  If $S$ is an inverse semigroup with zero, then $UF(E(S))$ is invariant under $S$~\cite{Exel}.  This led Exel to study the closure of $UF(E(S))$ in $\wh{E(S)}_0$.

Let $E$ be a semilattice with zero and $e\in E$.  A finite subset $F\subseteq e^{\downarrow}$ is a \emph{cover} of $e$ if $zf=0$ for all $f\in F$ implies that $ze=0$. Equivalently, $F$ is a cover of $e$ if $0\neq z\leq e$ implies that $zf\neq 0$ for some $f\in F$.  Note that the empty set if a cover of $0$.   Following Exel~\cite{Exel} (but using a reformulation of Lawson~\cite{Lawsontight}), we say that a filter $\mathcal F$ is \emph{tight} if $e\in \mathcal F$ and $F$ a cover of $e$ implies $F\cap \mathcal F\neq \emptyset$. Note that a tight filter must be proper because $\emptyset$ covers $0$.   We say that a  character $\theta$ is \emph{tight} if $\theta\inv(1)$ is a tight filter, or equivalently,
\[\theta(e)=\bigvee_{f\in F}\theta(f)\] whenever $F$ is a cover of $e$.  This is also equivalent to \[\prod_{f\in F}(\theta(e)-\theta(f))=0\] for each cover $F$ of $e$.  The space of tight characters is denoted $\wh E_T$.  Any ultrafilter is tight and Exel proved~\cite{Exel} $\wh E_T$ is the closure of $UF(E)$. In particular, $\wh{E(S)}_T$ is invariant for an inverse semigroup $S$ with zero.  We put $\mathscr U_T(S)=S\ltimes \wh{E(S)}_T$ and call it the \emph{universal tight groupoid} of $S$.

Our next goal is to give a presentation of the algebra of the universal tight groupoid under the assumption that $S$ is Hausdorff. We shall use without comment that the idempotents of a commutative ring form a generalized boolean algebra via $e\vee f=e+f-ef$, $e\wedge f=ef$ and $e\setminus f=e-ef$.  If $\psi\colon S\to A$ is a homomorphism to a $\Bbbk$-algebra $A$, we say that $\psi$ is \emph{tight} if $\psi(e)=\bigvee_{f\in F}\psi(f)$ whenever $F$ is a cover of $E$.

\begin{Prop}\label{p:presentation}
Let $S$ be a Hausdorff inverse semigroup and let $X\subseteq \wh{E(S)}$ be a closed invariant subspace. If $\mathscr G=\mathscr U(S)|_X$ and $\Bbbk$ is a commutative ring with unit, then \[\Bbbk\mathscr G\cong \Bbbk S/\langle \prod_{i=1}^n(e-e_i)\mid e_i\leq e, D(e)\cap D(e_1)^c\cap\cdots \cap D(e_n)^c\cap X=\emptyset\rangle.\]
\end{Prop}
\begin{proof}
Since $\mathscr G$ is a closed full subgroupoid of $\mathscr U(S)$, with unit space a union of orbits, there is a surjective homomorphism $F\colon \Bbbk\mathscr U(S)\to \Bbbk\mathscr G$ given by $F(\psi)=\psi|_{\mathscr G\skel 1}$.  Let $I=\ker F$.  We claim that $I$ is generated as an ideal by the $\chi_U$ such that $U\subseteq \wh {E(S)}$ is compact open and $U\cap X=\emptyset$.  Indeed, if $\p\in \ker I$, then because $\mathscr U(S)$ is Hausdorff, we have that $\p\inv(\Bbbk\setminus 0)$ is compact open and hence $U=\dom(\p\inv(\Bbbk\setminus 0))$ is compact open.  Moreover, $\p=\p\ast \chi_U$.  If $x\in U$, then $x=\dom(g)$ with $\p(g)\neq 0$. Since $X$ is invariant and $\p|_{\mathscr G\skel 1}=0$, we conclude that $x\notin X$. Thus $U\cap X=\emptyset$.  This proves the claim.

If $U\subseteq \wh{E(S)}$ is compact open with $U\cap X=\emptyset$, then $U=\bigcup_{i=1}^nB_i$ where $B_i$ are basic compact open subsets of $\wh {E(S)}$ (which necessarily satisfy $B_i\cap X=\emptyset$) and hence $\chi_U=\bigvee_{i=1}^n \chi_{B_i}$.  We deduce that $I$ is generated by the $\chi_B$ with $B$ a basic compact open subset of $\wh{E(S)}$ with $B\cap X=\emptyset$.  Such a basic neighborhood is of the form $B=D(e)\cap D(e_1)^c\cap\cdots \cap D(e_n)^c$ where $e_1,\ldots, e_n\leq e$.  Then $\chi_B=\prod_{i=1}^n (\chi_{D(e)}-\chi_{D(e_i)})$.  Under the isomorphism $\Bbbk S\to \Bbbk\mathscr U(S)$, we have that $\chi_B$ is the image of $\prod_{i=1}^n(e-e_i)$ and so the proposition follows.
\end{proof}

As a corollary, we obtain the following result.

\begin{Cor}\label{c:tightpres}
Let $S$ be a Hausdorff inverse semigroup with zero and let $\Bbbk$ be a commutative ring with unit.  Then
\begin{align*}
\Bbbk \mathscr U_T(S) &\cong \Bbbk S/\langle e-\bigvee F\mid F\ \text{covers}\ e\rangle\\
&\cong \Bbbk S/\langle \prod_{f\in F}(e-f)\mid F\ \text{covers}\ e\rangle.
\end{align*}
Hence the map $S\to \Bbbk\mathscr U_T(S)$ given by $s\mapsto \chi_{(s,D(s^*s)\cap \wh{E(S)}_T)}$ is the universal tight homomorphism from $S$ into a $\Bbbk$-algebra.
\end{Cor}
\begin{proof}
If $e_1,\ldots, e_n$ (where possibly $n=0$) is a cover of $e$, then $D(e)\cap D(e_1)^c\cap\cdots\cap D(e_n)^c$ cannot contain a tight character.  Conversely, if $e_1,\ldots, e_n$ is not a cover of $e$, then there exists $z$ with $0\neq z\leq e$ and $ze_i=0$ for $i=1,\ldots,n$.  Let $\mathcal F$ be an ultrafilter containing $z$ (such exists by Zorn's lemma).   Then $e\in \mathcal F$ and $e_1,\ldots, e_n\notin \mathcal F$.  Thus $\chi_{\mathcal F}\in D(e)\cap D(e_1)^c\cap\cdots\cap D(e_n)^c\cap \wh{E(S)}_T$ because ultrafilters are tight.  The result follows from Proposition~\ref{p:presentation}.
\end{proof}

Note that since the empty set is a cover of $0$, it follows that $\Bbbk S\to \Bbbk\mathscr U_T(S)$ factors through the contracted semigroup algebra $\Bbbk_0S$.

Since graph inverse semigroups are $0$-$E$-unitary and hence Hausdorff, it is immediate from the corollary and~\cite[Proposition~3.10]{JonesLawson} that if $S$ is a graph inverse semigroup of a directed graph in which the in-degree of every vertex is finite and at least $2$, then $\Bbbk \mathscr U_T(S)$ is isomorphic to the Leavitt path algebra corresponding to the same graph.

\subsection{Groupoids of germs from a dynamical viewpoint}
Let $S$ be an inverse semigroup acting on a locally compact Hausdorff space $X$.  We characterize the orbits and effectiveness of $\mathscr G=S\ltimes X$ in terms of the dynamics of the action of $S$ on $X$.

The first proposition shows that the orbits of $S$ and $S\ltimes X$ are the same.
\begin{Prop}\label{orbits}
Let $S$ be an inverse semigroup acting on a space $X$.  If $x\in X$, then $\mathcal O_x=\{sx\mid s\in S, x\in X_{s^*s}\}$.
\end{Prop}
\begin{proof}
This is immediate from the fact that $[s,x]$ is an arrow if and only if $x\in X_{s^*s}$ and that $\dom([s,x])=x$ and $\ran([s,x])=sx$.
\end{proof}

We say that the action of $S$ on $X$ is \emph{minimal} if there are no proper, non-empty closed $S$-invariant subspaces, or equivalently, each orbit under $S$ is dense.  We then have the following corollary of Proposition~\ref{orbits}.

\begin{Cor}\label{c:minimalactionofinv}
Let $S$ be an inverse semigroup acting on a space $X$.  Then $S\ltimes X$ is minimal if and only if the action of $S$ on $X$ is minimal.
\end{Cor}

Next we consider effectiveness.

\begin{Prop}\label{p:effectiveprop}
Let $S$ be an inverse semigroup acting on a space $X$.  Then $S\ltimes X$ is effective if and only if $X_s=\mathrm{Int}(\mathrm{Fix}(s))$ for all $s\in S$.
\end{Prop}
\begin{proof}
Let $\mathscr G=S\ltimes X$.  Recall that $X_s\subseteq \mathrm{Int}(\mathrm{Fix}(s))$ always is true.  Suppose that $\mathscr G$ is effective and let $x\in \mathrm{Int}(\mathrm{Fix}(s))$.  Note that $[s,y]\in \Is(\mathscr G)$ for all $y\in \mathrm{Fix}(s)$.  Let $U$ be an open neighborhood of $x$ contained in $\mathrm{Fix}(s)$.  Then $(s,U)$ is an open neighborhood of $[s,x]$ contained in $\Is(\mathscr G)$ and hence is contained in $\mathscr G\skel 0$.  In particular, $[s,x]$ is an identity and so there is an idempotent $e\leq s$ with $x\in X_e$.  Thus $x\in X_s$.

Conversely, suppose that $X_s=\mathrm{Int}(\mathrm{Fix}(s))$ and that $[s,x]\in \mathrm{Int}(\Is(\mathscr G))$. Let $(s,U)$ be a basic neighborhood of $[s,x]$ contained in $\Is(\mathscr G)$.  Then $sy=\ran([s,y])=\dom([s,y])=y$ for all $y\in U$ and so $x\in U\subseteq \mathrm{Fix}(s)$.  Thus $x\in \mathrm{Int}(\mathrm{Fix}(s))=X_s$ and so $x\in X_e$ for some $e\leq s$.  But then $[s,x]=[e,x]$ is an identity and so $\mathscr G$ is effective.
\end{proof}

\begin{Cor}\label{c:faithfulimplieseffective}
Let $S$ be an inverse semigroup acting faithfully on a space $X$ such that the $X_e$ with $e\in E(S)$ form a basis for the topology on $X$.  Then $S\ltimes X$ is effective.
\end{Cor}
\begin{proof}
This follows from Proposition~\ref{p:effectiveprop} and Proposition~\ref{p:faithful}.
\end{proof}

Proposition~\ref{p:effectiveprop} has a simpler formulation for $E$-unitary and $0$-$E$-unitary inverse semigroups.

\begin{Cor}\label{c:topologicallyfree}
Suppose that $\theta\colon S\to I_X$ is an action of an inverse semigroup on a space $X$.  If $S$ is $E$-unitary, or $S$ is $0$-$E$-unitary and $\theta(0)=0$, then $S\ltimes X$ is effective if and only if $\mathrm{Int}(\mathrm{Fix}(s))=\emptyset$ for all $s\in S\setminus E(S)$.
\end{Cor}
\begin{proof}
We handle just the $0$-$E$-unitary case, as the other case is simpler.  Note that if $S$ is $0$-$E$-unitary and if $X_0=\emptyset$, then $X_s=\emptyset$ for any $s\in S\setminus E(S)$.  The result is now immediate from Proposition~\ref{p:effectiveprop}.
\end{proof}

We now have the following theorem as a consequence of Corollary~\ref{simplehaus}.

\begin{Thm}\label{t:simplefromaction}
Let $S$ be an inverse semigroup and $\Bbbk$ a field.  Suppose that $S$ acts on a Hausdorff space $X$ with a basis of compact open sets such that $X_s$ is clopen for all $s\in S$. Setting $\mathscr G=S\ltimes X$, one has $\Bbbk \mathscr G$ is simple if and only if the action of $S$ is minimal and $X_s=\mathrm{Int}(\mathrm{Fix}(s))$ for all $s\in S$.
\end{Thm}
\begin{proof}
We have that $\mathscr G$ is Hausdorff by Proposition~\ref{p:hausdorffcondition}.  In light of Corollary~\ref{c:minimalactionofinv} and  Proposition~\ref{p:effectiveprop}, the result follows from Corollary~\ref{simplehaus}.
\end{proof}

\subsection{Simplicity of contracted inverse semigroup algebras}
If $S$ is a non-trivial inverse semigroup and $\Bbbk$ is a field, then $\Bbbk S$ is never simple because $\Bbbk$ is a homomorphic image via the mapping $s\mapsto 1$ for $s\in S$.  But if $S$ is an inverse semigroup with zero, then the contracted semigroup algebra $\Bbbk_0S$ can be simple.  It was observed by Munn that a necessary condition is that $S$ be congruence-free, since if $\equiv$ is a non-trivial congruence on $S$, then there is an induced surjective homomorphism $\Bbbk_0S\to \Bbbk_0[S/{\equiv}]$ which has a non-trivial kernel.  But even congruence-free inverse semigroups with zero can have non-simple contracted semigroup algebras.  For instance, the polycyclic inverse monoid on $2$ generators~\cite{Lawson} is congruence-free and its algebra has as a quotient the Leavitt path algebra associated to a rose with $2$ petals.  Munn asked~\cite{MunnAlgebraSurvey} for a characterization of when a congruence-free inverse semigroup with zero has a simple contracted semigroup algebra. We provide an answer to this question under the additional assumption that the inverse semigroup is Hausdorff.  In particular, this applies to $0$-$E$-unitary inverse semigroups.  We also show that the universal tight groupoid of a congruence-free Hausdorff inverse semigroup is always simple.

Recall that an inverse semigroup $S$ with zero is called \emph{$0$-simple} if it contains no proper, non-zero ideals, i.e, $SsS=S$ for all $s\neq 0$.
An inverse semigroup~$S$ is called \emph{fundamental} if every non-trivial congruence
identifies some pair of idempotents. A semilattice $E$ with zero is called
 \emph{$0$-disjunctive} if for all $0<e<f$, there exists $0<e'<f$ such
 that $ee'=0$.  It is well known~\cite{petrich} that an inverse semigroup $S$ with zero is congruence-free if and only if it is $0$-simple, fundamental and $E(S)$ is $0$-disjunctive.  Moreover, $S$ is fundamental if and only if the centralizer of $E(S)$ is $E(S)$~\cite{Lawson}, or equivalently,  $s^*es=e$ for all idempotents $e\leq s$ implies $s\in E(S)$. If every maximal subgroup of $S$ is trivial, then $S$ is fundamental.

\begin{Lemma}\label{l:ultrafilter}
Let $S$ be an inverse semigroup and suppose that  $S$ is fundamental and $E(S)$ is $0$-disjunctive.  Then $S\ltimes UF(E(S))$ is effective.
\end{Lemma}
\begin{proof}
Let $K_e=UF(E(S))\cap D(e)$ for $e\in E(S)$.  Then it is easy to see that the $K_e$ with $e\in E(S)$ form a basis for the topology on $UF(E(S))$ (cf.~\cite{Lawsontight}).  Indeed, if $\mathcal F$ is an ultrafilter with $e\in \mathcal F$ and $e_1,\ldots, e_n\notin \mathcal F$, then by maximality it follows, for $i=1,\ldots, n$, that $e_ie_i'=0$ for some $e_i'\in \mathcal F$.  Taking, $e'=ee_1'\cdots e_n'$, we have $e'\in \mathcal F$, $e'\leq e$ and $e'e_i=0$ for $i=1,\ldots, n$. Thus $K_{e'}\subseteq D(e)\cap D(e_1)^c\cap\cdots \cap D(e_n)^c\cap UF(E(S))$ with $\chi_{\mathscr F}\in K_{e'}$.

 Therefore, by Corollary~\ref{c:faithfulimplieseffective}, it suffices to prove that $S$ acts faithfully on $UF(E(S))$.  First note that $e\mapsto K_e$ is injective. Indeed, first observe that $K_e=\emptyset$ if and only if $e=0$, since any non-zero idempotent is contained in an ultrafilter by Zorn's lemma.  Suppose that $K_e=K_f$ with $0\neq e,f\in E(S)$.  Without loss of generality, assume $e\nleq f$. Then $K_{ef}=K_e$ and $ef< e$.  Also $ef\neq 0$ because $K_e=K_{ef}$ is non-empty. By the definition of $0$-disjunctive, there exists $0<e'<e$ such that $efe'=0$.  If $\mathcal F$ is an ultrafilter containing $e'$, then $e\in \mathcal F$ and $ef\notin\mathcal F$.  This contradicts $K_e=K_{ef}$ and hence the assumption that $K_e=K_f$.  It follows that if $\theta\colon S\to I_{UF(E(S))}$ is the action homomorphism, then $\theta$ is idempotent separating.  
As every idempotent separating homomorphism from a fundamental inverse semigroups is injective, we conclude that $S$ acts faithfully on $UF(E(S))$.  This completes the proof.
\end{proof}

As a corollary, we obtain an effectiveness result for $\mathscr U_T(S)$.

\begin{Cor}\label{c:tightiseffective}
Let $S$ be a fundamental inverse semigroup such that $E(S)$ is $0$-disjunctive.  Suppose, moreover, that $\mathscr U_T(S)$ is Hausdorff (e.g., if $S$ is Hausdorff).  Then $\mathscr U_T(S)$ is effective.
\end{Cor}
\begin{proof}
Let $X=\wh{E(S)}_T$ and suppose that $\theta\in X$ belongs to $\mathrm{Int}(\mathrm{Fix}(s))$.  Then $\theta=\lim \theta_{\alpha}$ where $\{\theta_{\alpha}\}_{\alpha\in D}$ is a net in $UF(E(S))$.  Since $\theta$ is an interior point of $\mathrm{Fix}(s)$, $\theta_{\alpha}\in \mathrm{Int}(\mathrm{Fix}(s))$ for all $\alpha$ sufficiently large. By Lemma~\ref{l:ultrafilter} and Proposition~\ref{p:effectiveprop} applied to $UF(E(S))$ we conclude that $\theta_{\alpha}\in X_s$ for $\alpha$ sufficiently large.  But $X_s$ is closed by Proposition~\ref{p:hausdorffcondition} and thus $\theta\in X_s$. We conclude that $\mathscr U_T(S)$ is effective by Proposition~\ref{p:effectiveprop}.
\end{proof}

Next we prove that if $S$ is $0$-simple, then $\mathscr U_T(S)$ is minimal.

\begin{Prop}\label{p:minimality}
Let $S$ be a $0$-simple inverse semigroup.  Then $\mathscr U_T(S)$ is minimal.
\end{Prop}
\begin{proof}
We must show that all orbits of $S$ on $X=\wh{E(S)}_T$ are dense. Since $UF(E(S))$ is dense in $X$, it suffices to show that each orbit contains $UF(E(S))$ in its closure. Let $\theta\in X$ and $\p\in UF(E(S))$. Let $X_e\cap X_{e_1}^c\cap\cdots \cap X_{e_n}^c$ be a basic neighborhood of $\p$.  Since $\mathcal F=\p\inv(1)$ is an ultrafilter and doesn't contain the $e_i$, there exists $e'\in E(S)$ with $\p(e')=1$, $e'\leq e$ and $e'e_i=0$ for $i=1,\ldots, n$ (cf.~the proof of Lemma~\ref{l:ultrafilter}).  Suppose that $\theta(f)=1$.  Since $S$ is $0$-simple, we have $f=se't$ with $s,t\in S$.  Put $z=e'tf$.  Then $sz=f$ and $zf=z$ and so $z^*z=z^*zf=fz^*z=szz^*z=sz=f$ and $e'zz^*=e'e'tfz^*=e'tfz^*=zz^*$ and so $zz^*\leq e'$. Therefore $\theta\in D(f)=D(z^*z)$ and \[(z\theta)(e')\geq (z\theta)(zz^*)=\theta(z^*zz^*z)=\theta(f)=1.\]  Thus $z\theta\in X_{e'}\subseteq X_e\cap X_{e_1}^c\cap\cdots \cap X_{e_n}^c$.  We conclude that $\ov{\mathcal O_{\theta}}\supseteq \ov{UF(E(S))}=X$, as required.
\end{proof}

The next corollary is one of our principal applications to inverse semigroups.

\begin{Cor}\label{c:simplicitytight}
Let $S$ be a congruence-free inverse semigroup and $\Bbbk$ a field. Suppose, moreover, that $\mathscr U_T(S)$ is Hausdorff (e.g., if $S$ is Hausdorff).  Then $\Bbbk \mathscr U_T(S)$ is simple.
\end{Cor}
\begin{proof}
Since congruence-free inverse semigroups are fundamental, $0$-simple and have $0$-disjunctive semilattices of idempotents, this is immediate from Corollary~\ref{simplehaus}, Corollary~\ref{c:tightiseffective} and Proposition~\ref{p:minimality}.
\end{proof}

We are now ready to characterize Hausdorff inverse semigroups with zero whose contracted semigroup algebras are simple, which is another main result of this paper. Let us say that an inverse semigroup $S$ with zero is \emph{tight}, if $0\neq e\in E(S)$ and $F$ a cover of $e$ implies $e\in F$.  Equivalently, $S$ is tight if each proper principal character of $S$ is tight, that is, $\mathscr U_0(S)=\mathscr U_T(S)$.  Notice that $S$ is tight if and only if $E(S)$ is tight.

\begin{Thm}\label{t:simpleinvalg}
Let $S$ be Hausdorff inverse semigroup with zero and $\Bbbk$ a field.  Then the contracted semigroup algebra $\Bbbk_0S$ is simple if and only if $S$ is congruence-free and tight.
\end{Thm}
\begin{proof}
Suppose first that $\Bbbk_0S$ is simple. We already observed that $S$ must be congruence-free.  By Corollary~\ref{simplehaus} and the isomorphism $\Bbbk_0S\cong \Bbbk\mathscr U_0(S)$, we must have that $\mathscr U_0(S)$ is minimal.  Since $\wh{E(S)}_T$ is a closed invariant subspace, we deduce that all proper characters of $S$ are tight and hence $S$ is tight.  Conversely, if $S$ is tight, then because the principal characters are dense and tight, we deduce that $\wh{E(S)}=\wh{E(S)}_T$ and hence $\mathscr U_0(S)=\mathscr U_T(S)$.  Therefore, we have $\Bbbk_0S\cong \Bbbk\mathscr U_0(S)\cong \Bbbk\mathscr U_T(S)$ and the result follows from Corollary~\ref{c:simplicitytight}.
\end{proof}

\begin{Cor}\label{c:simpleeunit}
Let $S$ be a $0$-$E$-unitary inverse semigroup and $\Bbbk$ a field.  Then the contracted semigroup algebra $\Bbbk_0S$ is simple if and only if $S$ is congruence-free and tight.
\end{Cor}

Let us consider an example.

\begin{Example}
If $X$ is a set with $|X|\geq 2$, then the polycyclic inverse monoid $P_X$~\cite{Lawson} is $0$-$E$-unitary and congruence-free.  If $|X|<\infty$, then $\{x^*x\mid x\in X\}$ covers $1$ and so $P_X$ is not tight.  Moreover, $\Bbbk\mathscr U_T(P_X)$ is the Leavitt algebra associated to a graph with one vertex and $|X|$ loops.  If $|X|$ is infinite, then $P_X$ is tight.  For if $ww^*$ is an idempotent (with $w$ a word over $X$) and if $wx_1(wx_1)^*,\ldots, wx_n(wx_n)^*$ are idempotents strictly below $ww^*$, then choosing $x\in X$ different from the first letters of $x_1,\ldots, x_n$, we have that $0\neq wx(wx)^*<ww^*$ and $wx(wx)^*(wx_i)(wx_i)^*=0$.  Thus $P_X$ is tight and therefore, $\Bbbk P_X$ is simple for any field $\Bbbk$ (as is well known).
\end{Example}

\subsection{Primitivity and semiprimitivity of inverse semigroup algebras}
In this section, we apply our results on primitivity and semiprimitivity of groupoid algebras to inverse semigroups.  We first give our promised example for which $\Bbbk$-density differs from density.

\begin{Example}\label{examplecliff}
Let $S$ be the following inverse monoid.  It consists of elements $\{a,e,0,1,2,3,\ldots\}$ where $e,0,1,2,3,\ldots$ are idempotents, $0$ is a multiplicative zero, $e$ is the identity, $ij=0$ for $i\neq j>0$, $a^2=e$ and $ai=i=ia$ for all $i\geq 0$.

Each character of $E(S)$ is principal. The universal groupoid $\mathscr U(S)$ is isomorphic to a group bundle with unit space $E(S)$ where $E(S)$ is given the topology of the one-point compactification of the discrete space $E(S)\setminus \{e\}$.  The isotropy group at $e$ is $\{a,e\}$ and is trivial at all other objects.  A subset containing $a$ is open if and only if contains all but finitely many elements of $E(S)$.  In particular, all neighborhoods of $e$ and $a$ intersect non-trivially and so $\mathscr U(S)$ is not Hausdorff.  Notice that $E(S)\setminus \{e\}$ is dense in $\mathscr G\skel 0$.  However, it is not $\Bbbk$-dense for any commutative ring with unit $\Bbbk$.  Let $U=\{a\}\cup E(S)\setminus\{e\}$.  Then $U\in \Bis_c(\mathscr U(S))$ and hence $\p=\chi_{E(S)}-\chi_U = \delta_e-\delta_a$ belongs to $\Bbbk \mathscr U(S)$.  There is no element $g\in \mathscr U(S)\skel 1$ with $\p(g)\neq 0$ and $\dom(g)\in E(S)\setminus \{e\}$.  Thus $E(S)\setminus \{e\}$ is not $\Bbbk$-dense.
\end{Example}

Our next proposition shows that the principal characters of an inverse semigroup are $\Bbbk$-dense for any base ring $\Bbbk$.

\begin{Prop}\label{p:idempotentsstronglydense}
Let $S$ be an inverse semigroup and $\Bbbk$ a commutative ring with unit.  Then the set of principal characters of $E(S)$ is $\Bbbk$-dense in $\mathscr U(S)$.
\end{Prop}
\begin{proof}
Recall that there is an isomorphism $\theta\colon \Bbbk S\to \Bbbk \mathscr U(S)$ given by sending $s\in S$ to $\chi_{(s,D(s^*s))}$. Suppose that $\p$ is the image under $\theta$ of $\sum_{i=1}^nc_is_i$ with $c_1,\ldots,c_n\in \Bbbk\setminus \{0\}$. Without loss of generality we may assume that $s_1$ is maximal among $s_1,\ldots,s_n$ in the natural partial order.  Then the arrow $[s_1,\chi_{s_1^*s_1}]$ belongs to $(s_1,D(s_1^*s_1))$ but not to $(s_i,D(s_i^*s_i))$, for $i\neq 1$, by maximality of $s_1$.  Thus $\p([s_1,\chi_{s_1^*s_1}])\neq 0$ and $\dom([s_1,\chi_{s_1^*s_1}])= \chi_{s_1^*s_1}$.  This proves that the principal characters are $\Bbbk$-dense.
\end{proof}

Next we recover Domanov's theorem~\cite{Domanov} as a special case of our groupoid results.

\begin{Cor}[Domanov]
Let $S$ be an inverse semigroup and $\Bbbk$ a commutative ring with unit.  Suppose that $\Bbbk G_e$ is semiprimitive for each idempotent $e\in E(S)$. Then $\Bbbk S$ is semiprimitive.  This applies, in particular, if $\Bbbk$ is a field of characteristic $0$ that is not algebraic over $\mathbb Q$.
\end{Cor}
\begin{proof}
If $\chi_e$ is the principal character associated to $e\in E(S)$, then $G_e$ is the isotropy group at $\chi_e$.  Since the principal characters are $\Bbbk$-dense in $\mathscr U(S)$ by Proposition~\ref{p:idempotentsstronglydense}, the result follows from Theorem~\ref{t:semiprimitivity} and the isomorphism $\Bbbk S\cong \Bbbk\mathscr U(S)$.
\end{proof}

We now give an example showing that $\Bbbk$-density cannot be relaxed to just density in Theorem~\ref{t:semiprimitivity}.

\begin{Example}\label{e:needkdense}
Let $S$ be the inverse monoid from Example~\ref{examplecliff} and let $\mathbb F_2$ be the $2$-element field.  The set of principal characters associated to the non-identity idempotents is dense in $\wh{E(S)}$ and the corresponding isotropy groups are trivial.  Thus $\mathbb F_2G_x$ is semiprimitive for $x$ in a dense subset of $\mathscr U(S)\skel 0$.  But $\mathbb F_2S$ is a commutative ring with unit and the element $e-a$ is nilpotent (because $(e-a)^2=e-2a+e=0$). Thus $\mathbb F_2S$ is not semiprimitive.  This shows that for non-Hausdorff groupoids, we need to work with $\Bbbk$-density rather than just density.
\end{Example}

A semilattice $E$ is \emph{pseudofinite}~\cite{MunnSemiprim} if, for all $e\in E$, the set of elements strictly below $e$ is finitely generated as a lower set.  In~\cite[Proposition~2.5]{mygroupoidalgebra}, it was shown that this is equivalent to the principal characters being isolated points of $\wh{E}$.  Indeed, if $e_1,\ldots, e_n$ generate the lower set of elements below $e$, then $D(e)\cap D(e_1)^c\cap\cdots\cap D(e_n)^c=\{\chi_e\}$.  The following corollary is the main result of~\cite{MunnSemiprim}.

\begin{Cor}[Munn]
Let $S$ be an inverse semigroup and $\Bbbk$ a commutative ring with unit.  If $E(S)$ is pseudofinite, then $\Bbbk S$ is semiprimitive if and only if $\Bbbk G_e$ is semiprimitive for each idempotent $e$.
\end{Cor}
\begin{proof}
The principal characters form a $\Bbbk$-dense set of isolated points in $\mathscr U(S)$ and their isotropy groups are precisely the maximal subgroups.  Corollary~\ref{c:isolatedsemi} and the isomorphism $\Bbbk S\cong \Bbbk\mathscr U(S)$ yield the required result.
\end{proof}

As another corollary, we obtain the following result for inverse semigroup algebras. Recall that an inverse semigroup $S$ (with zero) is \emph{($0$-)bisimple} if it contains a unique (non-zero) $\mathscr D$-class, that is, the principal characters (except $\chi_0$) of $\mathscr U(S)$ belong to a single orbit.

\begin{Cor}
Let $S$ be a ($0$-)bisimple inverse semigroup with maximal subgroup $G$ of its unique (non-zero) $\mathscr D$-class and let $\Bbbk$ be a field.  If $\Bbbk G$ is primitive, then so is the (contracted) semigroup algebra of $S$.  The converse holds if $E(S)$ is pseudofinite.
\end{Cor}
\begin{proof}
By ($0$-)bisimplicity, the principal characters form a $\Bbbk$-dense orbit.  They are isolated points if $E(S)$ is pseudofinite. The result follows from Theorem~\ref{t:primitivity}.
\end{proof}

In particular, we obtain the main result of~\cite{Munnprimitive}.

\begin{Cor}
Let $S$ be a $0$-bisimple inverse semigroup with trivial maximal subgroup of its unique non-zero $\mathscr D$-class.  Then the contracted semigroup algebra of $S$ is primitive over any field.
\end{Cor}

Munn constructed inverse semigroups with zero whose contracted semigroup algebras are simple over $\mathbb F_p$ (hence primitive) such that none of their maximal subgroups has a semiprimitive algebra~\cite{Munntwoexamples} over $\mathbb F_p$. His examples have the further  property that none of the isotropy groups of $\mathscr U(S)$ have semiprimitive algebras over $\mathbb F_p$ as is easily checked.

\end{document}